\documentclass{article}
\usepackage{amsmath}
\usepackage{amssymb,fullpage}

\usepackage{graphicx}
\usepackage{subfigure}
\usepackage{color}
\usepackage{enumerate}
\usepackage{cite}
\usepackage{graphicx}
\usepackage{multirow}
\usepackage{xr}

\newcommand{\bea}{\begin{eqnarray}}
\newcommand{\eea}{\end{eqnarray}}
\newcommand{\beas}{\begin{eqnarray*}}
\newcommand{\eeas}{\end{eqnarray*}}
\newcommand{\leftm}{\left[\begin{array}}
\newcommand{\rightm}{\end{array}\right]}
\newcommand{\reals}{\mbox{$\mathbb R$}}
\newcommand{\symm}{\mbox{$\mathbb S$}}

\newcommand{\ones}{\textbf{1}}
\newcommand{\zeros}{\textbf{0}}

\def\mL{{\mbox{$\mathcal{L}$}}}
\def\mA{{\mbox{$\mathcal{A}$}}}

\newcommand{\minimize}{\text{minimize}}
\newcommand{\st}{\text{subject to}}
\newcommand{\tr}{\textbf{Tr}}
\newcommand{\diag}{\textbf{diag}}
\newcommand{\Diag}{\textbf{Diag}}
\newcommand{\mb}{\mathbf}
\newcommand{\rank}{\textbf{rank}}
\newcommand{\sign}{\textbf{sign}}
\newcommand{\blkdiag}{\textbf{blkdiag}}
\newcommand{\ceil}[1]{\left \lceil{#1}\ \right \rceil }

\newcommand{\eg}{\textit{e.g.}}
\newcommand{\ie}{\textit{i.e.}}

\newtheorem{thm}{Theorem}[section]
\newtheorem{cor}[thm]{Corollary}

\newtheorem{lem}[thm]{Lemma}
\newtheorem{definition}[thm]{Definition}

\newtheorem{Assumption}[thm]{Assumption}

\newenvironment{proof}{\paragraph*{Proof:}}{\hfill$\square$}

\newcommand*\samethanks[1][\value{footnote}]{\footnotemark[#1]}

\usepackage{algorithm,algpseudocode}
\algnewcommand{\Inputs}{%
	\State \textbf{Inputs:}
}
\algnewcommand{\Initialize}{%
	\State \textbf{Initialize:}
}
\algnewcommand{\Outputs}{%
	\State \textbf{Outputs:}
}
\begin{document}
%

%

\title{ADMM for combinatorial graph problems}

\author{ Chuangchuang Sun\thanks{This work was done at Technicolor Research in the summer of 2017.}, Yifan Sun\samethanks, Ran Dai }


\maketitle

\begin{abstract}
  We investigate a class of general combinatorial graph problems, including MAX-CUT and community detection, reformulated as quadratic objectives over nonconvex constraints and solved via the alternating direction method of multipliers (ADMM).
  We propose two reformulations: one using vector variables and a binary constraint, and the other further reformulating the Burer-Monteiro form for simpler subproblems.
  Despite the nonconvex constraint, we prove the ADMM iterates converge to a stationary point in both formulations, under mild assumptions.
  Additionally, recent work suggests that in this latter form, when the matrix factors are wide enough, local optimum with high probability is also the global optimum.
  To demonstrate the scalability of our algorithm, we include results for MAX-CUT, community detection, and image segmentation  benchmark and simulated examples.

\end{abstract}

\section{INTRODUCTION} 

A number of important  problems arising in graph applications can be formulated as a linear binary optimization problem, for different choices of $C$:
\begin{equation}
\begin{array}{ll}
\minimize & x^TCx\\
\st & x \in \{+1,-1\}^n\\
\end{array}
\label{eq:main}
\end{equation}
Here, $C$ is an $n\times n$ symmetric matrix, where $n$ is the number of nodes in an undirected graph.
For different choices of $C$, problem \eqref{eq:main} may correspond to the   MAX-CUT problem in combinatorics, the 2-community detection problem in machine learning, or many other  combinatorial  graph problems.

Exactly solving combinatorial problems of this nature is a very difficult problem (it is in general NP-Hard), and even the best branch-and-bound methods do not scale well beyond a few thousands of nodes.

\subsection{Notation}
Define $\symm^n$ the space of $n\times n$ symmetric matrices.
For an undirected graph with $n$ nodes, we define $A\in \symm^n$ the adjacency matrix of the graph, where $A_{ij} >0 $ is the strictly positive weight of the edge between nodes $i$ and $j$,  and $A_{ij} = 0$ if there is no edge between $i$ and $j$. Define $\mb 1$ as the vector of all 1's.

\subsection{Community detection}
The problem of community detection is to identify node clusters in undirected graphs that are more likely to be connected with each other than with nodes outside the cluster.
This prediction  is useful in many graphical settings, such as interpreting online communities through social network or linking behavior, interpreting molecular or neuronal structure in biology or chemistry,  finding disease sources in epidemiology, and many more.
There are many varieties and methodologies in this field, and it would be impossible to list them all, though a very comprehensive overview is given in \cite{fortunato2016community}.

The stochastic binary model \cite{holland1983stochastic} is one of the simplest generative models for this application. Given a graph with $n$ nodes and parameters $0 < q < p < 1$, the model partitions the nodes into two communities, and generates an edge between nodes in a community with probability $p$ and nodes in two different communities with probability $q$. Following the analysis in \cite{abbe2016exact}, we can define  $C = \frac{p + q}{2}\mb1 \mb1^T - A$, and the solution  to \eqref{eq:main} gives a solution to the community detection problem with sharp recovery guarantees.



\subsection{MAX-CUT}
While the community detection problem aims to find two densely connected clusters in a graph, the MAX-CUT problem attempts to find two clusters with dense interconnections.
For an undirected graph with $n$ nodes and (possibly weighted) edges between some pairs of nodes, a \textit{cut} is defined as a partition of the $n$ nodes, and the \textit{cut value} is the sum of the weights of the severed edges. The MAX-CUT problem is to find the cut in a given graph that gives the maximum cut value.
When $C = (A - \diag(A\mb 1))/4$, the solution to \eqref{eq:main} is exactly the maximum cut.

Many methods exist to find the MAX-CUT of a graph.
In general, combinatorial methods can be solved using branch-and-bound schemes, using a linear relaxation of \eqref{eq:main} as a bound  \cite{barahona1988application,de1995exact}, where the binary constraint is replaced with $0 \leq (x+1)/2 \leq 1$. Historically, these ``polyhedral methods" were the main approach to find exact solutions of the MAX-CUT problem. Though this is an NP-Hard problem, if the graph is sparse enough, branch-and-bound converges quickly even for very large graphs \cite{de1995exact}. However, when the graph is not very sparse, the linear relaxation is loose, and finding efficient branching mechanisms is challenging, causing the algorithm to run slowly.

The MAX-CUT problem can also be approximated by one pass of the linear relaxation (with bound
$\frac{f_{\text{relax}}}{f_{\text{exact}}} \geq 2 \times \#$edges) \cite{poljak1994expected}.
A tighter approximation can be found with the semidefinite relaxation (see next section), which is also used for better bounding in branch-and-bound techniques \cite{helmberg1998solving, rendl2007branch, burer2008finite, bao2011semidefinite}.
In particular, the rounding algorithm of \cite{goemans1995improved} returns a feasible $\hat x$ given optimal $Z$, and is shown in expectation to satisfy $\frac{x^TCx}{\hat x^TC\hat x} \geq 0.878$. For this reason, the semidefinite relaxation for problems of type \eqref{eq:main} are heavily studied (\eg \cite{poljak1995recipe, helmberg2000semidefinite, fujie1997semidefinite}).

\subsection{Semidefinite relaxation}
To find the semidefinite relaxation of problem \eqref{eq:main}, we first reformulate it equivalently  using a change of variables $Z = xx^T$:
\begin{equation}\label{eq:UNn_MISDP}
\begin{aligned}
 &\underset{Z\in\symm^n}{\minimize}& &\tr(CZ)  \\
&\st& &\diag(Z) = 1  \\
& & &Z \succeq \textbf{0}  \\
& & &\rank(Z) \le 1.
\end{aligned}
\end{equation}
We refer to this formulation as the rank constrained semidefinite program (RCSDP). Without the rank constraint, \eqref{eq:UNn_MISDP} is the usual semidefinite (convex) relaxation used in MAX-CUT approximations; we refer to the semidefinite relaxation   as the SDR. Solving the SDR exactly can be done in polynomial time (\eg using an interior point method); however, the per-iteration complexity can be large ($O(n^6)$ for interior point method and  $O(n^3)$ for most first-order methods) limiting its scalability.

 Consequently, a reformulation based on matrix factorization has been proposed by~\cite{burer2003nonlinear,burer2005local}, which solves the SDR with a factorization $Z = RR^T$, $R\in \mathbb{R}^{n\times r}$:
\begin{equation}\label{eq:FormulationXxy0SymFact}
\begin{aligned}
&\underset{R\in \reals^{n\times r}}{\minimize}& &\tr(R^T CR)  \\
&\st& &\diag(RR^T) = \ones \\
\end{aligned}
\end{equation}
which we will call the factorized SDP (FSDP). Note that FSDP is not convex; however, because the search space is compact, if $r$ is large enough such that $r(r+1)/2 >  n$, then the global optimum of the FSDP (\ref{eq:FormulationXxy0SymFact}) is the global optimum of the SDR
~\cite{pataki1998rank,barvinok1995problems}.
Furthermore, a recent work~\cite{boumal2016non} shows that almost all local optima of FSDP are also global optima, suggesting that any stationary point of the FSDP is also a reasonable approximation of \eqref{eq:main}.

To solve (\ref{eq:FormulationXxy0SymFact}),~\cite{burer2003nonlinear} put forward an algorithm based on augmented Lagrangian method to find local optima of (\ref{eq:FormulationXxy0SymFact}); unsurprisingly, due to the later obsrvation of Boumel et. al, this method empirically found global optima almost all of the time.
However, solving \eqref{eq:FormulationXxy0SymFact} is still numerically burdensome; in the augmented Lagrangian term, the objective is quartic in $R$, and is usually solved using an iterative numerical method, such as L-BFGS.

Another way of finding a low-rank solution to the SDR is to use the nuclear norm as a surrogate for the rank penalty \cite{candes2009exact, recht2010guaranteed, udell2016generalized}. This method is popular because of its theoretical guarantees and experimental success.
A difficulty with nuclear norm minimization is in choosing the penalty parameter; in general this is done through cross-validation.
Also, when the rank function appears as a hard constraint in the optimization problem, the corresponding bound for the surrogate model is ambiguous.

\subsection{ADMM for nonconvex optimization}
We propose using the Alternagting Direction Method of Multipliers (ADMM) to solve vector and SDP-based matrix reformulations of \eqref{eq:main}, as a scalable alternative to the prior methods mentioned here.
Recently, the Alternating Direction Method of Multipliers (ADMM) \cite{gabay1976dual, glowinski1975approximation, lions1979splitting, eckstein1992douglas} has been widely applied and investigated in various fields; see the survey~\cite{boyd2011distributed} and the references therein.
Specifically, it is favored because it is easy to apply to a distributed framework, and has been shown to converge quickly in practice.
For convex optimization problems, ADMM can be shown to converge to a global optima in several ways: as a series of contractions of monotone operators \cite{eckstein1992douglas} or as the minimization of a global potential function \cite{boyd2011distributed}.

However, in general there is a lack of theoretical justification for ADMM on nonconvex problems  despite its good numerical performance.
Almost all works concerning ADMM on nonconvex problems investigate when nonconvexity is in the objective functions (\ie \cite{hong2016convergence,wang2015global,li2015global,magnusson2016convergence}, and also
\cite{lu2017nonconvex,xu2012alternating} for matrix factorization)
Under a variety of assumptions (\eg~convergence or boundedness of dual objectives)  they are shown to convergence to a KKT stationary point.
%
%
%

In comparison, relatively fewer works deal with nonconvex constraints. ~\cite{jiang2014alternating} tackles polynomial optimization problems by minimizing a general objective over a spherical constraint $\|x\|_2 = 1$,  ~\cite{huang2016consensus} solves general QCQPs, and \cite{shen2014augmented} solves the low-rank-plus-sparse matrix separation problem. In all cases, they show that all limit points are also KKT stationary points, but do not show that their algorithms will actually converge to the limit points. In this work, we investigate a class of nonconvex constrained problems, and show with much milder assumptions that the sequence always converges to a KKT stationary point.







\section{VECTOR FORM}
We begin by considering a simple reformulation of \eqref{eq:main} using only vector variables
\begin{equation}
\label{eq:Formulation3}
\begin{array}{ll}
\underset{x,y\in \reals^n}{\minimize}& x^TCx \\
\st& y \in \{-1,1\}^n\nonumber\\
    & x = y.\nonumber
    \end{array}
\end{equation}
This is similar to the nonconvex formulations proposed in \cite{boyd2011distributed}, Chapter 9.1.

To solve \eqref{eq:Formulation3} via ADMM, we first form the augmented Lagrangian
\bea\label{eq:Formulation4}
\mL(x,y;\mu) = x^TCx + \langle {\mu},x-y\rangle + \frac{\rho}{2}\|x-y\|_F^2,
\eea
where $\mu\in\reals^{n}$ is the corresponding dual variable and $\rho > 0$ is the weighting factor associated with the penalty term.
Under the ADMM framework, problem in (\ref{eq:Formulation4}) can be solved by alternating between the two primal variables $x$ and $y$, and then updating the dual variable $\mu$. The ADMM for \eqref{eq:Formulation3} is summarized in Algorithm \ref{a:vector}.
Note that unlike typical ADMM, the penalty parameter $\rho$ in Algorithm \ref{a:vector} is increasing with the iteration index, which will facilitate the convergence.

\begin{algorithm}
	\caption{ADMM for solving \eqref{eq:Formulation3}}
	\begin{algorithmic}[1]
		\Inputs{$C\in \symm^n$, $\rho_0>0$, $\alpha>1$, tol $\epsilon > 0$}
		\Outputs {Local optimum $(x,y)$ of (\ref{eq:Formulation3})}
		\Initialize{$x^0, y^0;\mu^0$ as random vectors}
		
		\For{$k = 1 \hdots$}

		\State Update $y^k$ as the solution to
		
		\begin{equation}\label{eq:z1z2}
		\begin{array}{ll}
		\underset{y}{\min} & \|x^{k}-y+\frac{\mu^{k}}{\rho^{k}}\|_F^2 \\
		\text{s.t.} &  y \in \{-1,1\}^n
		\end{array}
		\end{equation}
		
		\State Update ${x^k}$ the minimizer of
			\begin{equation}\label{eq:xnew}
			x^TCx +  \frac{\rho^{k}}{2}\|x-y^{k+1} + \frac{\mu^k}{\rho^{k}}\|_F^2
			\end{equation}
		\State Update $\mu$ via
			\bea\label{eq:LambdaUpd}
			\mu^{k+1} &=& \mu^{k} + \rho^{k}({x}^{k+1} - {y}^{k+1})\nonumber\\
			\rho^{k+1} &=& \alpha\rho^k
			\eea

		\If{$\|x^{k}-y^k\| \leq \epsilon$}
		\State {\textbf{break}}
		\EndIf
		\EndFor
	\end{algorithmic}
	\label{a:vector}
\end{algorithm}



\subsection{Solution for Subproblems in ADMM}

For every step $k$ in ADMM of Algorithm  \ref{a:matrix}, the updating procedure is composed of three subproblems, expressed in (\ref{eq:z1z2})-(\ref{eq:LambdaUpd}).
The update for $y$ in \eqref{eq:z1z2} is just rounding
\[
y_i^{k+1} = \sign({x}^{k}+\mu^{k}/\rho^k)_i,
\]
and is computationally cheap.
However, updating $x$ is more cumbersome. The update of $x$ is the solution to a linear system
\bea\label{eq:OptCond1}
2Cx + \mu^k + \rho^k(x-y^{k+1}) = 0.
\eea
and
$x^k = (\rho^kI + 2C)^{-1}(-\mu^k + \rho^k y^{k+1})$.
If $\rho^k=\rho$ is constant, then one can avoid inverting at each iteration by storing $(\rho^kI + 2C)^{-1}$ and using matrix multiplication at each step.
For changing $\rho^k$, one can store the eigenvalue decomposition of $C = U_C\Lambda_C U_C^T$, and compute for each $\rho^k$
\[
(\rho^kI + 2C)^{-1} = U_C(2\Lambda_C + \rho^kI)^{-1}U_C^T.
\]
Although this formulation seems simple, when $n$ is large, the initial eigenvalue decomposition or matrix inverse can incur significant overhead.
Still, despite these limitations, this vector-based method can be shown to converge under very mild conditions.

\subsection{Convergence Analysis}
We will now show that Algorithm \ref{a:vector} converges under very mild conditions; specifically,  the augmented Lagrangian monotonically decreases.


\begin{definition}\label{def:L}
$x^TCx$ is $L_1-$smooth function, that is, any $x_1, x_2$, $\| 2Cx_1 - 2Cx_2\| \le L_1\|x_1 - x_2\|$.\\
\end{definition}

\begin{definition}\label{Definition:L_H}
$L_H\in\reals_{+}$ is the smallest number such that $-L_HI \preceq 2C$, where $2C$ is the Hessian of the objective function of \eqref{eq:main}. In other words, the Hessian is lower bounded.
\end{definition}

\begin{Assumption}\label{Assumption:Df}
The value $x^TCx$ is lower bounded over $x\in \{-1,1\}^n$.
\end{Assumption}

\begin{thm}\label{lemma:LagDecent}
Given the sequence $\{\rho^k\}$ such that
\[
(\rho^k)^2-L_H\rho^k - (\alpha+1)L_1^2 > 0,\rho^k > L_H, \alpha > 0
\]
under Algorithm \ref{a:vector} the augmented Lagrangian monotonically decreases
\[
\mathcal{L}(y^{k+1},x^{k+1};\mu^{k+1}) - \mathcal{L}(y^{k},x^{k};\mu^{k}) \le - c_1^k\|x^{k+1} - x^{k}\|_F^2
\]
with $c_1^k = \frac{\rho^k-L_H}{2} - \frac{(\alpha+1)L_1^2}{2\rho^{k}} >0$.
With Assumption \ref{Assumption:Df} and $\rho^k > L_1$, the augmented Lagrangian  $\mL({y}^{k},{x}^k;{\mu}^k)$ is lower bounded and convergent.
Therefore, $\{{y}^k,{x}^k,{\mu}^k\}$ converge to $\{{y}^*,{x}^*,{\mu}^*\}$ a KKT solution of \eqref{eq:Formulation3} and
${x}^* = {y}^*$.
\end{thm}
\begin{proof}
See section \ref{sec:vectorconvergence} in the appendix.
\end{proof}

\textit{Remark}: Convergence is also guaranteed under a constant penalty coefficient $\rho_k \equiv \rho^0$ ($\alpha=1$) satisfying $(\rho^0)^2-L_H\rho^0 - 2L_1^2 > 0$.
However, in implementation, we find empirically that increasing $\{\rho^k\}$ from a relatively small $\rho^0$  can encourage convergence to more useful global minima.

\begin{cor}\label{thm:Convergence}
With the prerequisites in Theorem \ref{lemma:LagDecent} satisfied and $(y^*,x^*;\mu^*)$ as the accumulation point, we have that\\
\begin{enumerate}
	\item[i.]$\mu_i^*x_i^* \ge -\rho^*$ with $x^*\in\{-1,1\}^n$.
	\item[ii.] ${x}^*$ is a stationary point of the optimization problem $\min_{{x}}x^TCx+({\mu}^*)^T{x}$.
\end{enumerate}
 In particular for a convex function objective $f(x) = x^TCx$ ($C \succeq 0$), ${\mu}^* = 0$ indicates that the constraint $x\in\{-1,1\}^n$ is automatically satisfied; that is, the minimizer of $f(x)$ is the exact minimizer of \eqref{eq:main}.


\end{cor}

\begin{proof}
See section \ref{sec:vectorconvergence} in the appendix.
\end{proof}

\section{PSD MATRIX FORM}

We now present our second reformulation of  \eqref{eq:main}, using ideas from the SDR and FSDP:
\begin{equation}
\label{eq:FormulationXxy1}
\begin{array}{cl}
\underset{X,Y , Z}{\minimize} & \tr(CZ)\\
\st & \diag(Z) = \mb 1\\
& Z = XY^T\\
& X = Y.
\end{array}
\end{equation}
Here the  matrix variables are $Z\in \symm^n$, and $X$, $Y\in \reals^{n\times r}$. The extra variable $Y$ and the constraint $X = Y$ are introduced to transform the quadratic constraint in \eqref{eq:FormulationXxy0SymFact} to  bilinear, which will simplify the subproblems significantly.
Two cases of $r$ are of interest.
When  $r = 1$, (\ref{eq:FormulationXxy1}) is  equivalent to  (\ref{eq:main}), with $X = x$. When, $r = \ceil {\sqrt{2n}}$, then $\frac{r(r+1)}{2} > n$ and the global optimum of (\ref{eq:FormulationXxy1}) is with high probability that of the SDR, based on the results of \cite{boumal2016non}.


To solve (\ref{eq:FormulationXxy1}) via ADMM, we  build the augmented Lagrangian function as
\begin{align}\label{eq:AugLag1}
\mL(Z,X,Y;&\Lambda_1,\Lambda_2) = \tr(CZ)  + \langle \Lambda_2,X-Y\rangle \nonumber\\
& +\langle \Lambda_1,Z - XY^T\rangle + \frac{\rho}{2}\|X-Y\|_F^2 \nonumber\\
  &+ \frac{\rho}{2}\|Z - XY^T\|_F^2
\end{align}
where $\Lambda_1\in\reals^{n\times n}$ and $\Lambda_2\in\reals^{n\times r}$ are the dual variables corresponding to the two coupling constraints, respectively. Similar as to the vector case, problem  (\ref{eq:FormulationXxy1}) is solved by alternating minimization between two primal variable sets, $Y$ and $(Z,X)$, and then updating the dual variable $\Lambda_1$ and $\Lambda_2$. This sequence is summarized in  Algorithm \ref{a:matrix} below.

\begin{algorithm}
	\caption{ADMM for solving \eqref{eq:FormulationXxy1}}
	\begin{algorithmic}[1]
		\Inputs{$C\in \symm^n$, $\rho_0>0$, $\alpha>1$, tol $\epsilon > 0$}
		\Initialize{$Z^0, X^0;\Lambda_1^0,\Lambda_2^0$ as random matrices}
		\Outputs {$Z$, $X=Y$}
		\For{$k = 1 \hdots$}
		\State Update ${Y^{k+1}}$ the minimizer of
			\begin{equation}\label{eq:xnewXxy}
			\|Z^{k} - X^{k}Y^T + \frac{\Lambda_1^k}{\rho^{k}}\|_F^2
			+ \|X^{k}-Y + \frac{\Lambda_2^k}{\rho^k}\|_F^2
			\end{equation}
		\State Update $(Z,X)^{k+1}$ as the solutions of
			\begin{equation}\label{eq:z1z2Xxy}
			\begin{array}{ll}
			\underset{X,Z}{\min} & \mL(Z,X,Y^{k+1};\Lambda_1^{k},\Lambda_2^{k};\rho^k) \\
			\text{s.t.} &  \diag(Z) = \mb 1
			\end{array}
			\end{equation}
			where $\mL$ is as defined in \eqref{eq:AugLag1}
		\State Update $\Lambda_1, \Lambda_2$ and $\rho$ via
			\bea\label{eq:LambdaUpdXxy}
			\Lambda_1^{k+1} &=& \Lambda_1^{k} + \rho^k({Z}^{k+1} - {X}^{k+1}({Y}^{k+1})^T)\nonumber\\
			\Lambda_2^{k+1} &=& \Lambda_2^{k} + \rho^k({X}^{k+1} - {Y}^{k+1})\nonumber\\
			\rho^{k+1} &=& \alpha\rho^{k}
			\eea
		\If{$\max \{\|X^{k}-Y^k\| ,\|Z^{k}-X^k(Y^k)^T\|\} \leq \epsilon$}
		\State {\textbf{break}}
		\EndIf
		\EndFor
	\end{algorithmic}
\label{a:matrix}
\end{algorithm}


\subsection{Solution for Subproblems in ADMM}



For notation ease, we omit the iteration index $k$. To update $Y$ in (\ref{eq:z1z2Xxy}),
we solve the first order optimality condition, which reduces to the following linear system:
\[Y = \left(\frac{1}{\rho}(\Lambda_1^TX+\Lambda_2) + Z^TX + X\right)(I+X^TX)^{-1}.\]
For $r = \ceil {\sqrt{2n}}$, the size of the matrix to be inverted is of the order $O(\sqrt{n})$; for $r = 1$, it is a scalar and inversion is trivial.

To update $(Z,X)$ in (\ref{eq:z1z2Xxy}), we similarly need to solve an equality-constrained quadratic problem. Define the linear map $\mathcal{A}(Z) = \diag(Z)$ selecting the diagonal of a symmetric matrix, and its adjoint operator $\mathcal{A}^*(\nu) = \Diag(\nu)$ producing a diagonal matrix from a vector. Then the optimality conditions of (\ref{eq:z1z2Xxy}) are

\begin{equation*}
\begin{array}{rcl}
 C - \mA^*(\nu) + \Lambda_1 + \rho(Z-XY^T) &=& 0\\
 \Lambda_2 -\Lambda_1Y + \rho(XY^T - Z)Y + \rho(X-Y) &=& 0\\
\mA(X) &=& \ones,
\end{array}
\end{equation*}
where $\nu\in\reals^m$  is the dual variable for the local constraint $\mA(X) = b$.

Given $Y$, solving for $\nu$ reduces to solving
\[
G\nu = \rho(b-\mA(DY^T)) + \mA[(C+\Lambda_1)(I + YY^T)]
\]
where
\[
D =  \frac{1}{\rho}(\Lambda_1Y -\Lambda_2) + Y , \quad
G = \mA^*(\mA(I + YY^T)).
\]
Since  $G$ is a diagonal positive definite matrix then finding $G^{-1}$ is computationally simple.
Then
\[
X = BY + D, \quad \text{ and } \quad
Z = XY^T + B,
\]
where
\[
B = -\frac{1}{\rho}(C-\mA^*(\nu) + \Lambda_1).
\]
Note that the complexity of all inversions are $O(r^{3})$, and are particularly simple if $r = 1$. In that case, the complexity is dominated by the matrix multiplications.




\subsection{Convergence Analysis}


\begin{Assumption}\label{Assumption:BoundedDualXxy}
The sequences $\{\tr(CZ^k)\}$ and $\{\Lambda_1^k,\Lambda_2^k\}$ are bounded in norm.
\end{Assumption}

\begin{thm}\label{lemma:LagDecentXxy}
If Assumption \ref{Assumption:BoundedDualXxy} holds, and given a sequence $\{\rho^k\}$ such that
\[
\rho^k > 0, \sum\frac{\rho^{k+1}}{(\rho^k)^2} < \infty, \text{ and } \sum\frac{1}{\rho^k} < \infty,
\]
then
 $\{Z^k, X^k, Y^k\}$ generated by Algorithm  \ref{a:matrix} will globally converge to a stationary point of (\ref{eq:FormulationXxy1}).
\end{thm}
\begin{proof}
  See section \ref{sec:matrixproof} in the appendix.
\end{proof}


\begin{cor}
	If $r \geq \ceil{\sqrt{2n}}$ and the  stationary point of Algorithm \ref{a:matrix} converges to a second order critical point of \eqref{eq:FormulationXxy1}, then it is globally optimal for the convex relaxation of \eqref{eq:FormulationXxy0SymFact}~\cite{boumal2016non}.
\end{cor}

Unfortunately, the extension of KKT stationary points to global minima is not yet known when $\frac{r(r+1)}{2} < n$ (\ie, $r = 1$). However, our empirical results suggest that even when $r = 1$, often a local solution to \eqref{eq:FormulationXxy1} well-approximates the global solution to \eqref{eq:main}.

\section{NUMERICAL RESULTS}

From the convergence proofs, we have shown that our algorithms converge to a KKT stability point, and if $r = \ceil{\sqrt{2n}}$, with high probability achieves the global solution of the SDR. In this section, we give empirical evidence that they are in addition good solutions to \eqref{eq:main}, our original combinatorial problem. We show this through three applications: community detection of the stochastic block model, MAX-CUT, and image segmentation.

We evaluate four methods for solving \eqref{eq:main}:
\begin{enumerate}
	\item SD: the solution to the SDR  rounded to a binary vector  using a  Goemans-Williamson style rounding \cite{goemans1995improved} technique;
	
	\item V: the binary vector reformulation  \eqref{eq:Formulation3} solved via ADMM;
	
	\item MR1: the matrix reformulation \eqref{eq:FormulationXxy1} with $r = 1$, solved via ADMM;
	\item MRR:  the matrix reformulation \eqref{eq:FormulationXxy1} with $r = \ceil{\sqrt{2n}}$, solved via ADMM, and rounded to a binary vector using a nonsymmetric version of the Goemans-Williamson style rounding \cite{goemans1995improved} technique.
\end{enumerate}

\paragraph{Rounding} For both the SDP and MRR methods, we need to round the solutions to a vector.
 For the SDP method, we first do an eigenvalue decomposition $X = Q\Lambda Q^T$ and form  a factor $F = Q\Lambda^{1/2}$ where the diagonal elements of $\Lambda$ are in decreasing magnitude order. Then we scan $k = 1,\hdots, n$ and find $x_{k, t} = \sign(F_k z_t)$ for trials $t = 1,\hdots  10$. Here, $F_k$ contain the first $k$ columns of $F$, and each element of $z_t$ is drawn i.i.d from a normal Gaussian distribution.
We keep $x_r = x_{k,t}$ that minimizes $x_r^TCx_r$.
For the MRR method, we repeat the procedure using a factor $F = U\Sigma^{1/2}$ where $X=U\Sigma V^T$ is the SVD of $X$.
For MR1 and V, we simply take $x_r = \sign(x)$  as the binary solution.

\paragraph{Computer information}
The following simulations are performed on a Dell Poweredge R715 with an 
AMD Opteron(tm) processor, with
64 GB of RAM and
24 cores. It is running 
Ubuntu 16.04 (Xenial)
with Matlab 2017a.

\subsection{Community detection}
In this section we evaluate $C$ as defined for community detection.
Figure \ref{f:sbm_evolution} gives a sample evolution for each method with $n = 2500$, $m = n/2$, and $p = 10q = 0.04$. Although the vector method is simple and has a fast per-iteration rate, it has significant initial overhead and slow overall progress. In comparison, the two matrix methods MR1 and MRR are significantly faster. The SDR method has a slow per iteration rate, but suggests good answers in only a few iterations. However, this plot does not take into account rounding overhead time (see figure \ref{f:sbm_runtime}), which is significant for both the SDR and MRR methods.  

\begin{figure}
	\begin{center}
	\includegraphics[width=3in]{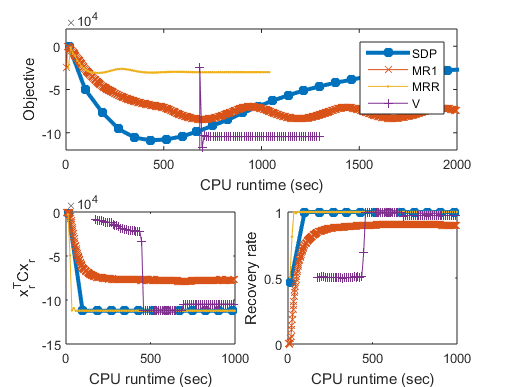}
	\end{center}
	\caption{Sample evolution for $n = 2500$, showing the objective value $\tr(CX)$, best objective value using $x_r = \sign(x)$, and recovery rate using $x_r$. }
	\label{f:sbm_evolution}
\end{figure}

\begin{figure}
\begin{center}
	\includegraphics[width=3in]{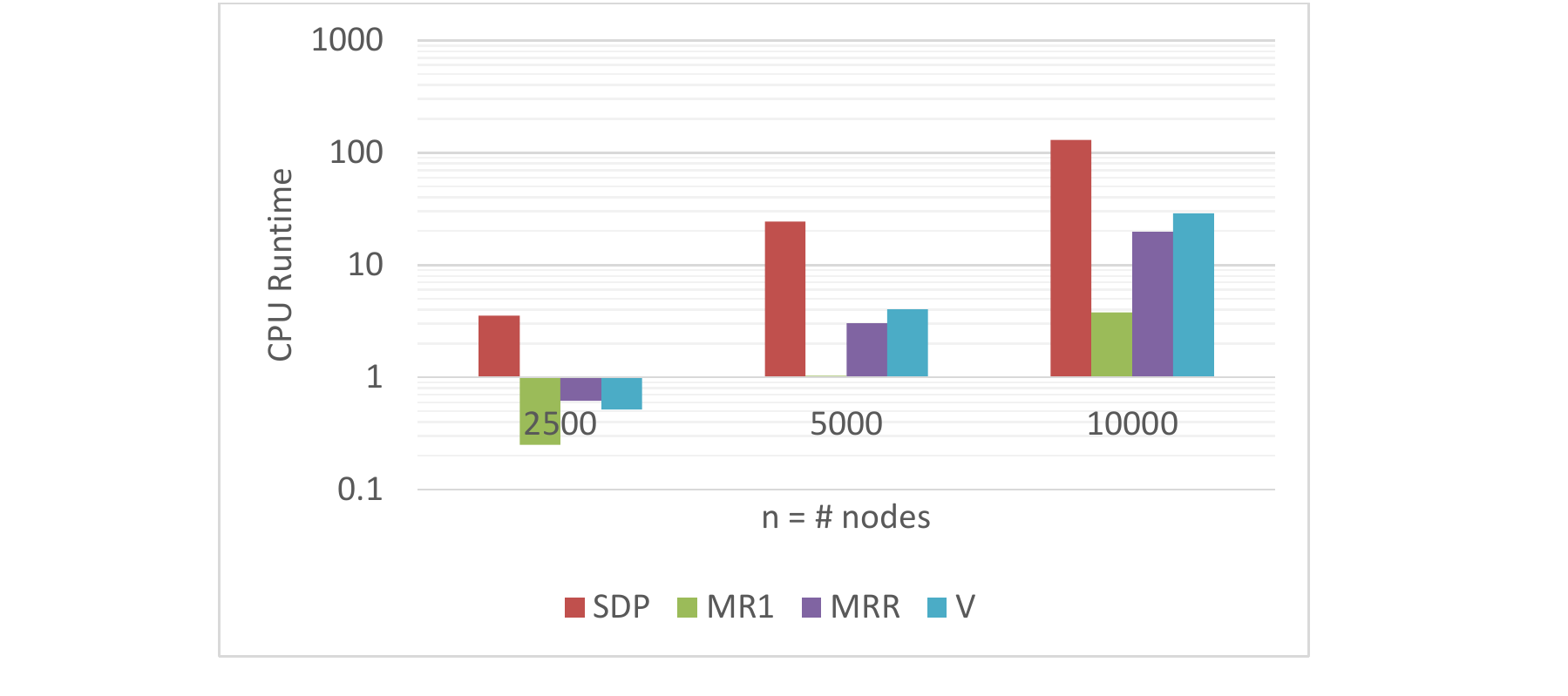}
	\includegraphics[width=3in]{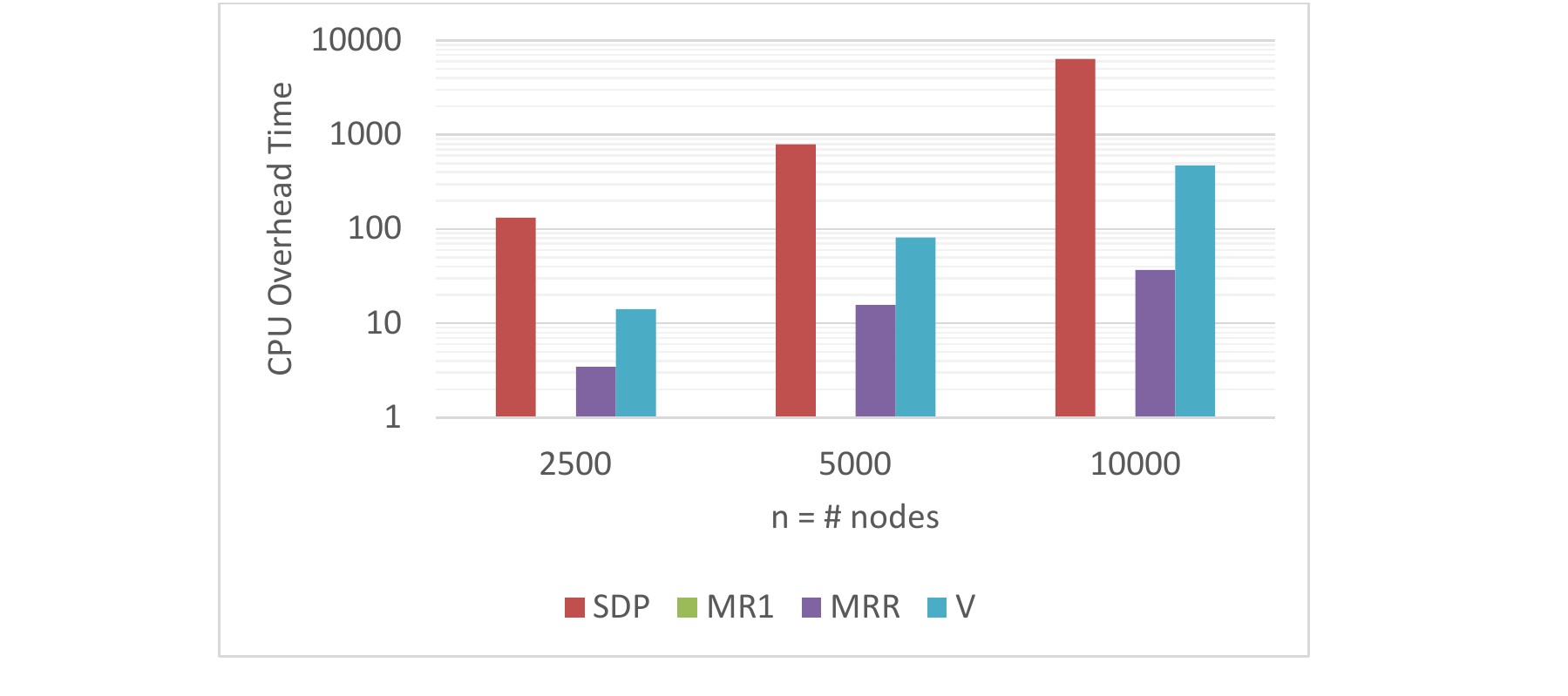}
	\end{center}
	\caption{Top: CPU per-iteration runtime for four methods. Bottom: CPU runtime for overhead operations (rounding for SDP and MRR and initial matrix inversion or SVD for V). MR1 gives a vector output and has no  additional overhead.
	}
	\label{f:sbm_runtime}
\end{figure}

Table \ref{t:sbm_recovery} gives the average recovery rate for each method on the community detection problem for the stochastic block model. The results are surprisingly consistent, and suggest that solving \eqref{eq:main} using our approaches all have great promise for general community detection problems.

	\begin{table}
		\begin{center}
			\begin{tabular}{ c c c | c c c c }\hline
				$n$ & $m$ & $p$   &S & V & MR1 & MRR   \\\hline\hline
				1000 & 100 & 0.1    &1& 1& 1 & 1  \\
				1000 & 500 & 0.1    &1&1 & 1 & 1  \\
				2500 & 250 & 0.04   &1& 1& 1 & 1  \\
				2500 & 1250 & 0.04  &1&1 & 1 & 1.00  \\
				5000 & 500 & 0.02   &1& 1 & 1 & 1.00  \\
				5000 & 2500 & 0.02  &1& 1 & 1 & 1.00 \\
				10000 & 1000 &0.01  &1& 1 & 1 & 1  \\
				10000 & 5000 &0.01  &1&1  &1  &1.00   \\\hline
			\end{tabular}
			\caption{\textbf{Correct community recovery rate for stochastic block model.} Result after 10 iterations for S, MR1, and MRR methods and 50 iterations for V, averaged over 10 trials for each method. $n = $ number of nodes in graph. $m = $ number of nodes in smaller of two communities. $p$ = percentage of  edges within communities, and $q = p/10$ the percentage of edges  between communities.  1 = perfect recovery, 1.00 = rounds to 1.}
		\end{center}
	\label{t:sbm_recovery}
\end{table}

\subsection{MAX-CUT}

Table 
\ref{t:maxcut}
gives the best MAX-CUT values using best-of-random-guesses and our approaches over four examples from the 
7th DIMACS Implementation Challenge in 2002.\footnote{See \texttt{http://dimacs.rutgers.edu/Workshops/7thchallenge/.} Problems downloaded from \texttt{http://www.optsicom.es/maxcut/}}
Our solutions are very close to the best known solutions.

	\begin{table}
		\begin{center}
			{\small
			\begin{tabular}{ c|c c c  c  }\hline
				&g3-15 & g3-8 & pm3-15-50&pm3-8-50\\\hline
				$n$&3375&512	&3375&512	\\
				\# edges &20250&3072&20250&3072\\
			sparsity &	0.18\% & 1.2\% & 0.18\% & 1.2\%\\\hline
			BK &281029888&41684814	&2964	&454\\
			R&12838418&4816306&	170		&62		\\
		MR1&	255681256	&36780180	&1990	&312\\	
			V&202052290 &8213762 &2016 &330\\\hline
			\end{tabular}}
		
		
	
			\caption{MAX-CUT values for graphs from the 7th DIMACS Challenge.
				BK = best known. R = best of 1000 random guesses. MR1 = matrix formulation, $r = 1$. V = vector formulation.}
			
			\label{t:maxcut}
		\end{center}
\end{table}

\subsection{Image segmentation}
Both community detection and MAX-CUT can be used in image segmentation, where each pixel is a node and the similarity between pixels form the weight of the edges. Generally, solving \eqref{eq:main} for this application is not preferred, since the number of pixels in even a moderately sized image is extremely large. However, because of our fast methods, we successfully performed image segmentation on several thumbnail-sized images, in figure \ref{f:imageseg}.

The $C$ matrix is composed as follows. For each pixel, we compose two feature vectors: $f_c^{ij}$ containing the RGB values and $f_p^{ij}$ containing the pixel location. Scaling $f_c^{ij}$ by some weight $c$, we form the concatenated feature vector $f^{ij} = [f_c^{ij}, cf_p^{ij}]$, and form the weighted adjacency matrix as the squared distance matrix between each feature vector $A_{(ij), (kl)} = \|f^{ij} - f^{kl}\|_2^2$.
For MAX-CUT, we again form $C = A - \Diag(A\mb 1)$ as before. For community detection, since we do not have exact $p$ and $q$ values, we use an approximation as $C = a\mb1 \mb1^T-A$ where $a = \frac{1}{n^2}\mb1^TA\mb1$ the mean value of $A$. Sweeping $C$ and $\rho_0$, we give the best qualitative result in figure \ref{f:imageseg}.

\begin{figure}[!h]
\begin{center}
	\includegraphics[width=3in]{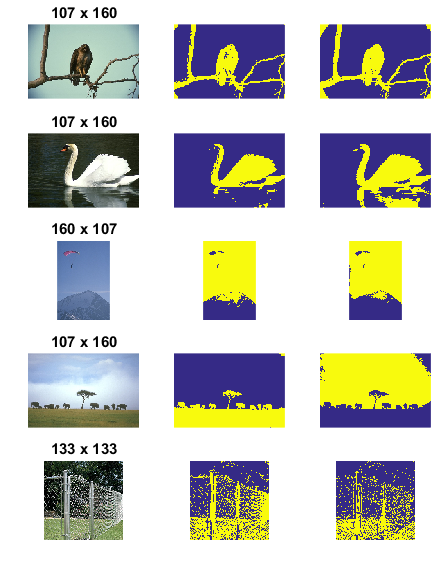}
	\end{center}
	\caption{Image segmentation. The center (right) column are the best MAX-CUT (community detection) results.}
	\label{f:imageseg}
\end{figure}

\section{CONCLUSION}

We present two methods for solving  quadratic combinatorial problems using ADMM on two reformulations. Though the problem has a nonconvex constraint, we give convergence results to KKT solutions under much milder conditions than previously shown. From this, we give empirical solutions to  several graph-based combinatorial problems, specifically MAX-CUT and community detection; both can can be used in additional downstream applications, like image segmentation.


\bibliographystyle{alpha}


\newcommand{\etalchar}[1]{$^{#1}$}

\begin{appendix}

\section{Convergence analysis for vector form}\label{sec:vectorconvergence}
We first prove the following lemma.

\begin{lem}\label{lemma:Lambda}
	If Assumption \ref{Assumption:Df} holds, for  two adjacent iterations of Algorithm \ref{a:vector} we have
	 \bea\label{eq:OptCond3}
	\|\mu^{k+1} - \mu^{k}\|_2^2 \le L_1^2 \|x^{k+1} - x^{k}\|_2^2.
	\eea
\end{lem}
\begin{proof}
	Recall the first order optimality conditions for $x$ in (\ref{eq:OptCond1})
	\bea\label{eq:OptCond0}
	2Cx^{k+1} + \mu^k + \rho^k(x^{k+1}-y^{k+1}) = 0.
	\eea
	Combining with (\ref{eq:LambdaUpd}), we get
	\bea\label{eq:OptCond2}
	2Cx^{k+1} + \mu^{k+1} = 0.
	\eea
	As (\ref{eq:OptCond2}) holds for each iteration $k$, combining it with the definition of $L_1$ proves the lemma.
\end{proof}

\textit{Remark}: We update $x$ after $y$ so that we can apply the optimality condition (\ref{eq:OptCond2}).

Next we will show that the augmented Lagrangian \eqref{eq:Formulation4} is monotonically decreasing and lower bounded.
\begin{lem}\label{lemma:auglag}
	The augmented Lagrangian \eqref{eq:Formulation4} is monotonically decreasing; that is,
		
	\begin{equation}\label{eq:D5}
	\mathcal{L}({y}^{k+1},{x}^{k+1};{\mu}^{k+1};\rho^{k+1}) - \mathcal{L}({y}^{k},{x}^{k};{\mu}^{k};\rho^{k})  
	\le c_1^k\|{x}^{k+1}-{x}^k\|_F^2.
	\end{equation}
	where $c_1^k = \frac{\rho^k-L_H}{2} - \frac{(\alpha+1)L_1^2}{2\rho^{k}} >0$.
\end{lem}
\begin{proof}
	For the update of $y$ in (\ref{eq:z1z2}), we have
	\bea\label{eq:D1}
		\mathcal{L}(y^{k+1},x^k;{\mu}^k;\rho^{k}) - \mathcal{L}({y}^{k},{x}^k;{\mu}^k;\rho^{k}) \le 0,
	\eea
	based on the fact that problem (\ref{eq:z1z2}) is globally minimized.
	Using $\{y, \mu, \rho\} = \{y^{k+1}, \mu^k, \rho^k\}$,  for the update of (\ref{eq:xnew}), we have
	\begin{eqnarray}
		\mathcal{L}({y},{x}^{k+1};{\mu};\rho) - \mathcal{L}({y},{x}^k;{\mu};\rho) 
		&\le& \langle \nabla_{{x}}\mL({y},{x}^{k+1};{\mu}; \rho),{x}^{k+1}-{x}^k\rangle  
		-\frac{\rho^k-L_H}{2}\|{x}^{k+1}-{x}^k\|_2^2  \nonumber\\
		&\le& -\frac{\rho^k-L_H}{2}\|{x}^{k+1}-{x}^k\|_2^2,\label{eq:D2}
	\end{eqnarray}
	where the first inequality follows that $\mL(y,{x};\mu;\rho)$ is strongly convex with respect to $x$ given $\rho > L_H$, and the second inequality follows from the optimality condition of (\ref{eq:xnew}).

	In addition, using $\{x,y\} = \{x^{k+1},y^{k+1}\}$ we have
	\begin{eqnarray}
	\mathcal{L}({y},{x};{\mu}^{k+1};\rho^{k+1}) - \mathcal{L}({y},{x};{\mu}^{k};\rho^{k})
	&=&\langle {\mu}^{k+1}-{\mu}^{k},{x} - {y}\rangle \nonumber + \frac{\rho^{k+1}-\rho^{k}}{2}\|{x} - {y}\|_F^2\\
	&=& \frac{\alpha+1}{2\rho^{k}}\|{\mu}^{k+1}-{\mu}^{k}\|_2^2 \nonumber\\
	\end{eqnarray}
	where the first equality follows the definition of $\mathcal{L}$ and the second follows (\ref{eq:LambdaUpd}).
	Combining with Lemma \eqref{lemma:Lambda}
	\begin{eqnarray}
		\mathcal{L}({y^{k+1}},{x^{k+1}};{\mu}^{k+1};\rho^{k+1}) - \mathcal{L}({y^{k+1}},{x^{k+1}};{\mu}^{k};\rho^{k}) 
	&\le& \frac{(\alpha+1)L_1^2}{2\rho^{k}}\|{x}^{k+1} - {x}^{k}\|_F^2,\label{eq:D3}
	\end{eqnarray}

	Moreover, it can be easily verified that
	\begin{eqnarray}
	\mathcal{L}(y^{k+1},x^{k+1};\mu^{k+1};\rho^{k+1}) - \mathcal{L}(y^{k},x^{k};\mu^{k};\rho^{k}) 
	&=&\mathcal{L}(y^{k+1},{x}^k;\mu^k;\rho^{k}) - \mathcal{L}(y^{k},x^{k};\mu^{k};\rho^{k}) \nonumber\\
	&&\;\;+\mathcal{L}(y^{k+1},x^{k+1};\mu^k;\rho^{k}) - \mathcal{L}(y^{k+1},x^{k};\mu^k;\rho^{k})\nonumber \\
	&&\;\;+\mathcal{L}(y^{k+1},x^{k+1};\mu^{k+1};\rho^{k+1}) \nonumber\\
	&&\;\;- \mathcal{L}(y^{k+1},x^{k+1};\mu^k;\rho^{k}).
	\label{eq:D4}
	\end{eqnarray}
	By incorporating (\ref{eq:D1}), (\ref{eq:D2}) and (\ref{eq:D3}), (\ref{eq:D4}) can be rewritten as
	
	\begin{multline}
	\mathcal{L}({y}^{k+1},{x}^{k+1};{\mu}^{k+1};\rho^{k+1}) - \mathcal{L}({y}^{k},{x}^{k};{\mu}^{k};\rho^{k})  \nonumber\\
	\le-\frac{\rho^{k}-L_H}{2}\|{x}^{k+1}-{x}^k\|_F^2 + \frac{(\alpha+1)L_1^2}{2\rho^{k}}\|{x}^{k+1} - {x}^{k}\|_F^2 \nonumber\\
	\le -\left(\frac{\rho^k-L_H}{2} - \frac{(\alpha+1)L_1^2}{2\rho^{k}}\right)\|{x}^{k+1}-{x}^k\|_F^2.
	\end{multline}

	With $(\rho^k)^2-L_H\rho^k - (\alpha+1)L_1^2>0,\rho^k>0$, we have $c_1 = \frac{\rho^k-L_H}{2} - \frac{(\alpha+1)L_1^2}{2\rho^{k}} >0$.

	\end{proof}
	
	\begin{lem}
	    If the objective of \eqref{eq:Formulation3} is lower-bounded over $x\in\{-1,1\}^n$ (assumption \eqref{Assumption:Df}), then the augmented Lagrangian \eqref{eq:Formulation4} is lower bounded.
	
	\end{lem}

	\begin{proof}
 For notation ease, we denote $f(x) = x^TCx$. From the $L_1$-Lipschitz continuity of  $\nabla_x f(x)$ , it follows that
	\bea\label{eq:LHUpp}
	f(y) \le f(x) + \langle \nabla f(x), y - x \rangle + \frac{L_1}{2}\|x - y\|_F^2
	\eea
	for any $x$ and $y$.
	By definition
	\begin{eqnarray}
	\mL({y}^{k},{x}^k; {\mu}^k;\rho^k) &=& f({x}^{k}) + \langle {\mu}^k,{x}^{k} - {y}^{k}\rangle
	 + \frac{\rho^k}{2}\|{x}^{k}-{y}^{k}\|_F^2 \nonumber \\
	&=&  f({x}^{k}) - \langle \nabla f(x^k),{x}^{k} - {y}^{k}\rangle 
	 + \frac{\rho^{k}}{2}\|{x}^{k}-{y}^{k}\|_F^2 \label{eq:CRCP92}\\
	&\ge&  f({y}^{k}) + \frac{\rho^{k}-L_1}{2}\|{x}^{k}-{y}^{k}\|_F^2 ,\label{eq:CRCP9}
	\end{eqnarray}
	where \eqref{eq:CRCP92} follows from (\ref{eq:OptCond2}) and \eqref{eq:CRCP9} follows from (\ref{eq:LHUpp}).
	From the boundedness of the objective (Assumption \ref{Assumption:Df}) and the assumption $\rho^k > L_1$, it follows that $\mL({y}^{k},{x}^k;{\mu}^k;\rho^k)$ is bounded below for all $k$.
\end{proof}

	Since the sequence $\{\mL({z}^{k},{x}^k;{\mu}^k)\}$ is monotonically decreasing and lower bounded, then the sequence $\{\mL({z}^{k},{x}^k;{\mu}^k)\}$ converges. Combining with Lemma \ref{lemma:Lambda} we have
	\bea
	\|{x}^{k+1} - {x}^{k}\| \to 0,
	\eea
	which indicates that the sequence $\{{x}^k\}$ converges to a limit point ($x^*$). Moreover, as $\{{x}^k\}$ converges, the sequence $\{{\mu}^{k}\}$ also converges to $\mu^*$ based on (\ref{eq:OptCond3}).
	From (\ref{eq:LambdaUpd}), ${\mu}^{k+1} - {\mu}^{k} = \rho^k ({x}^{k+1} - {y}^{k+1})$ and so  ${x}^{k+1} - {y}^{k+1} \to 0$; thus $\{{y}^{k}\}$ converges to a limit point ${y}^*$ and ${x}^*={y}^*$. Additionally, note that $x^* = y^*$ implies
	 $\nabla_x \mL(y^*,x^*;\mu^*;\rho) = \nabla_y \mL(y^*,x^*;\mu^*;\rho) = 0$ and $x^* = y^* \in \{-1,1\}$ are feasible.
Thus  Theorem \ref{lemma:LagDecent} is proven.

We now prove Corollary \ref{thm:Convergence}.
\begin{proof}
	i) At the accumulation point, the update of ${x}^*$ is expressed as
	\bea
	&{x}^* = {y}^{*} = \text{Proj}_{\{-1,1\}^n}({x}^{*}+\frac{u_i^*}{\rho^*}).
	\eea
	Since $x_i^* = \sign(\rho^*x_i^*+u_i^*)$, therefore
	\[
	x_i^*(\rho^*x_i^*+u_i^*) = \rho^*+x_i^*u_i^*\ge 0.
	\]
	
	ii) Given $x^* = y^*$, then from \eqref{eq:OptCond1} it follows that
	\bea\label{eq:z1z2Opt}
	\nabla_x f({x}^{*}) + {\mu}^*= 0,
	\eea
	or equivalently,
	\[
	x^* = \arg\min_x f(x) + x^T\mu^*.
	\]
	If, additionally, $\mu^* = 0$, then $x^*$ is also the minimizer of the unconstrained problem; that is
	\[
	\arg\min_x f(x) + x^T\mu^* = \arg\min_{x\in \{-1,1\}} f(x) = \arg\min_x f(x)
	\]
	suggesting that the combinatorial constraints $x\in \{-1,1\}$ are not necessary, and the convex relaxation is exact; therefore  $x^*$ is the global minimum of \eqref{eq:main}.
	
\end{proof}

\section{Convergence analysis for matrix form}\label{sec:matrixproof}

To simplify notation, we first collect the primal and dual variables $P^k = (Z,X,Y)^k$ and $D^k = (\Lambda_1,\Lambda_2)^k$.

\begin{lem}
	$\mL(P^k;D^k;\rho^{k})$ is bounded.
\end{lem}

\begin{proof}
Recall the definition of $\mL(P^k;D^k;\rho^{k})$

\begin{eqnarray}
\mL(P^k;D^k;\rho^{k}) &=& \tr(CZ^k) 
+ \frac{\rho^{k}}{2}\|Z^k - X^k(Y^k)^T + \frac{\Lambda_1^k}{\rho^k}\|_F^2 \nonumber\\
&& + \frac{\rho^k}{2}\|X^k-Y^k + \frac{\Lambda_2^k}{\rho^k}\|_F^2 
 - \frac{1}{2\rho^k}\left(\|\Lambda_1^k\|_F^2+\|\Lambda_2^k\|_F^2\right)\label{eq:Lbound}
 >-\infty
 \end{eqnarray}

where the inequality follows the boundedness of $\{\tr(CZ^k)\}$ and $\{\Lambda_1^k,\Lambda_2^k\}$ in Assumption \ref{Assumption:BoundedDualXxy}.
\end{proof}

\begin{lem}\label{lem:hessianbound}
	\[
	\nabla^2 \mL_Y \succeq \rho^k I
	\]
	and
		\[
	\nabla^2 \mL_{(X,Z)} \succeq \rho^k\left(1- \frac{\sqrt{\lambda_N^2+4\lambda_N}-\lambda_N}{2}\right)I.
	\]
\end{lem}

\begin{proof}
	
		Given the definition of $\mL$, we can see that the Hessian
		\[
		\nabla^2 \mL_Y = \rho^k\left(M+   I\right) \succeq \rho^k I
		\]
		where
		\[
		M = \blkdiag((X^k)^T(X^k),(X^k)^T(X^k),...)\succeq 0.
		\]
		
		For $(X,Z)$, we have
		\[\nabla^2_{(X,Z)}{\mL}_k = \rho^k\begin{bmatrix}
		I + NN^T & -N \\
		-N^T & I
		\end{bmatrix}
		\]
		where 
		\[
		N = \blkdiag((Y^{k+1})^T,\ldots,(Y^{k+1})^T )\in\reals^{nr\times n^2}.
		\]
		Note that for block diagonal matrices, $\|N\|_2 = \|Y^{k+1}\|_2$.
		Since $I\succ \zeros$ and its Schur complement $(I + NN^T) - NN^T \succ \zeros$, then  $\nabla^2_{(X,Z)}\tilde{\mL}_k \succ \zeros$ and equivalently $\lambda_{\min}(\nabla^2_{(X,Z)}{\mL}_k) > 0$.
		
		To find the smallest eigenvalue $\lambda_{\min}(\nabla^2_{(X,Z)}{\mL}_k)$, it suffices to find the largest $\sigma > 0$ such that
		\begin{equation}
		H_2 = (\rho^k)^{-1}\nabla^2_{(X,Z)}\tilde{\mL}_k - \sigma I 
		= \begin{bmatrix}
		(1-\sigma)I + NN^T & -N \\
		-N^T & (1-\sigma)I
		\end{bmatrix}
		\succeq\zeros.
		\end{equation}
		Using the same Schur Complement trick, we want to find the largest $\sigma > 0$ where $		(1-\sigma)I \succeq 0$ and
		\[H_3 = (1-\sigma)I + NN^T(1-(1-\sigma)^{-1})) \succeq 0.\]
		
		Since this implies $1-\sigma > 0$, then together with $\sigma>0$, it must be that $1-(1-\sigma)^{-1} \ < 0$.
		
		Therefore, defining $\lambda_N = \|Y^{k+1}\|^2_2$,
		\[\lambda_{\min}(H_3) = (1-\sigma) + \lambda_N(1-(1-\sigma)^{-1}).\]
		We can see that $(1-\sigma)\lambda_{\min}(H_3)$ is a convex function in $(1-\sigma)$, with two zeros at
		\[1-\sigma =  \frac{\pm\sqrt{\lambda_N^2+4\lambda_N}-\lambda_N}{2}.\]
		In between the two roots, $\lambda_{\min}H_3 <0$. Since the smaller root cannot satisfy $1-\sigma > 0$, we choose
	  \[\sigma_{\max} = 1- \frac{\sqrt{\lambda_N^2+4\lambda_N}-\lambda_N}{2} > 0\]
	 	as the largest feasible $\sigma$ that maintains $\lambda_{\min}(H_3)\geq 0$.
		
	
		As a result,
		\begin{eqnarray*}
			\lambda_{\min}(\nabla^2_{(X,Z)}\mL) &=& \rho^k\sigma_{\max}\nonumber\\
			&=&\rho^k\left(1- \frac{\sqrt{\lambda_N^2+4\lambda_N}-\lambda_N}{2}\right).
		\end{eqnarray*}

\end{proof}

We now prove Thm. \ref{lemma:LagDecentXxy}.

\begin{proof}
For the update of $Y$ in (\ref{eq:xnewXxy}), taking $\{D, \rho\} = \{D^k, \rho^k\}$,  we have

\begin{eqnarray}
\mathcal{L}(Z^{k},X^{k},Y^{k+1};D^k;\rho^k) - \mathcal{L}(P^k;D^k;\rho^k)
&\le& \langle \nabla_{Y}\mL(Z^{k},X^{k},Y^{k+1};D^k;\rho^k),{Y}^{k+1}-{Y}^k\rangle \nonumber\\
&&\quad -\frac{\lambda_{\min}(\nabla^2_{{Y}}\bar{\mL}_k)}{2}\|{Y}^{k+1}-{Y}^k\|_F^2  \nonumber\\
&\le& -\frac{\rho^{k}}{2}\|{Y}^{k+1}-{Y}^k\|_F^2 \label{eq:D2Xxy}
\end{eqnarray}
with $\nabla^2_{{Y}}\bar{\mL}_k = \nabla^2_{\text{vec}(Y)}\mL(Z^{k},X^{k},Y^{K+1};D^k;\rho^{k}) \succeq \rho^{k}I$.

which follows from the $\rho^k$-strong convexity of $\mL$ with respect to $Y$, and the optimality of $Y^{k+1}$.


For the update of $(Z,X)$ in (\ref{eq:xnewXxy}), we have
\begin{eqnarray}
\mathcal{L}(P^{k+1}; D^k;\rho^{k}) - \mathcal{L}(Z^{k},X^{k},Y^{k+1};D^k;\rho^{k})
&\le&\langle \nabla_{{Z}}\mL(P^{k+1};D^k;\rho^{k}),{Z}^{k+1}-{Z}^k\rangle 
 + \langle \nabla_{{X}}\mL(P^{k+1};D^k;\rho^{k}),{X}^{k+1}-{X}^k\rangle\nonumber\\
&& -\frac{\lambda_{\min}(\nabla^2_{(X,Z)}\tilde{\mL}_k)}{2}\big(\|{Z}^{k+1}-{Z}^k\|_F^2 
 +\|{X}^{k+1}-{X}^k\|_F^2\big)\nonumber\\
&\le&  -\frac{\lambda_{\min}(\nabla^2_{(X,Z)}\tilde{\mL}_k)}{2}\big(\|{Z}^{k+1}-{Z}^k\|_F^2
 + \|{X}^{k+1}-{X}^k\|_F^2\big),\label{eq:D1Xxy}
\end{eqnarray}
where the first inequality follows from the strong convexity of $\mL(Z,{X},Y^k;\Lambda_1^k,\Lambda_2^k)$ with respect to $(X,Z)$ with $\nabla^2_{(X,Z)}\tilde{\mL}_k = \nabla^2_{\text{vec}(X,Z)}\mL(P^{k+1};D^k;\rho^{k})$. In addition, the second inequality follows the optimality condition of (\ref{eq:z1z2Xxy}). Note that $\lambda_{\min}(\nabla^2_{(X,Z)}\tilde{\mL}) > 0$ given $\mL(Z,{X},Y^k;\Lambda_1^k,\Lambda_2^k)$ is strongly convex with regard to $(X,Z)$.
%
%

In addition for the update of the dual variables and the penalty coefficient, we have
\begin{eqnarray}
\mathcal{L}(P^{k+1};  D^{k+1};\rho^{k+1}) - \mathcal{L}(P^{k+1};D^{k};\rho^{k})
&=& \langle {\Lambda_1}^{k+1}-{\Lambda_1}^{k},Z^{k+1}-X^{k+1}(Y^{k+1})^T\rangle \nonumber\\
 &&+ \langle {\Lambda_2}^{k+1}-{\Lambda_2}^{k},X^{k+1}-Y^{k+1}\rangle \nonumber\\
 &&+\frac{\rho^{k+1}-\rho^k}{2}(\|Z^{k+1}-X^{k+1}(Y^{k+1})^T\|_F^2) \nonumber\\
&&+\frac{\rho^{k+1}-\rho^k}{2}(\|X^{k+1} - Y^{k+1}\|_F^2) \nonumber\\
&=&  \frac{\rho^{k+1}+\rho^k}{2(\rho^k)^2}\big(\|{\Lambda_1}^{k+1}-{\Lambda_1}^{k}\|_F^2  + \|{\Lambda_2}^{k+1}-{\Lambda_2}^{k}\|_F^2 \big)\label{eq:D3Xxy}
\end{eqnarray}
where the first equality follows the definition of $\mathcal{L}$, the second follows (\ref{eq:LambdaUpdXxy}).
Moreover, it can be easily verified that
\begin{multline}
\mathcal{L}(P^{k+1};D^{k+1};\rho^{k+1}) - \mathcal{L}(P^{k};D^{k};\rho^k) \\
=\mathcal{L}(Z^{k},X^{k},Y^{k+1};D^k;\rho^k) - \mathcal{L}(P^k;D^k;\rho^k) \\
+\mathcal{L}(P^{k+1};D^k;\rho^k) - \mathcal{L}(Z^{k},X^{k},Y^{k+1};D^k;\rho^k) \\
+\mathcal{L}(P^{k+1};D^{k+1};\rho^{k+1}) - \mathcal{L}(P^{k+1};D^{k};\rho^k).\label{eq:D4Xxy}
\end{multline}
By incorporating (\ref{eq:D1Xxy}), (\ref{eq:D2Xxy}) and (\ref{eq:D3Xxy}), (\ref{eq:D4Xxy}) can be rewritten as
\begin{multline}
\mathcal{L}(P^{k+1};D^{k+1};\rho^{k+1}) - \mathcal{L}(P^{k};D^{k};\rho^k) \\
\le-c_1^k\|{Z}^{k+1}-{Z}^k\|_F^2-c_2^k\|{X}^{k+1}-{X}^k\|_F^2\\
-c_3^k\|{Y}^{k+1}-{Y}^k\|_F^2\\
+\frac{\rho^{k+1}+\rho^k}{2(\rho^k)^2}\big(\|{\Lambda_1}^{k+1}-{\Lambda_1}^{k}\|_F^2 +
 \|{\Lambda_2}^{k+1}-{\Lambda_2}^{k}\|_F^2 \big)\label{eq:D5Xxy}.
\end{multline}
with $c_1^k = c_2^k = \frac{\lambda_{\min}(\nabla^2_{(Z,X)}\tilde{\mL}_k)}{2} > 0$ as given in Lemma \ref{lem:hessianbound} and $c_3^k = \frac{\rho^{k}}{2} > 0$.

By telescoping, the summation
\begin{eqnarray}\label{eq:D6Xxy}
\mL(P^K;D^K;\rho^K) - \mL(P^0;D^0;\rho^0)
&=&
\sum_{k=0}^{K-1}\mL(P^{k+1};D^{k+1};\rho^{k+1}) - \mL(P^k;D^k;\rho^k) \nonumber\\
&\leq&\sum_{k=0}^{K-1}\frac{\rho^{k+1}+\rho^k}{2(\rho^k)^2}\bigg(\|\Lambda_1^{k+1} - \Lambda_1^k\|_F^2 + \|\Lambda_2^{k+1} - \Lambda_2^k\|_F^2\bigg) \nonumber\\
&&-\sum_{k=0}^{K-1}c^k\bigg(\|{Z}^{k+1}-{Z}^k\|_F^2+\|{X}^{k+1}-{X}^k\|_F^2+\|{Y}^{k+1}-{Y}^k\|_F^2 \bigg),\nonumber
\end{eqnarray}
with $c^k = \min(c_1^k, c_2^k, c_3^k) > 0$.

For the first term,
\begin{eqnarray*}
	0&\leq& \sum_{k=0}^{K-1}\frac{\rho^{k+1}+\rho^k}{2(\rho^k)^2}\bigg(\|\Lambda_1^{k+1} - \Lambda_1^k\|_F^2 + \|\Lambda_2^{k+1} - \Lambda_2^k\|_F^2\bigg) \\
	&\leq& 4\sum_{k=0}^{K-1}\frac{\rho^{k+1}+\rho^k}{(\rho^k)^2}  B
\end{eqnarray*}


where from Assumption \ref{Assumption:BoundedDualXxy} we have that $\{\Lambda_1^k,\Lambda_2^k\}$ is bounded; that is there exists $B < \infty$ where $B \geq \max\{\|\Lambda_1^k\|_F^2,\|\Lambda_2^k\|_F^2\}$ for all $k$. Since additionally
 $\sum \frac{\rho^{k+1}+\rho^k}{2(\rho^k)^2} \le \sum\frac{\rho^{k+1}}{(\rho^k)^2} < +\infty$, then
 \begin{eqnarray*}
 0&\leq& \sum_{k=0}^{K-1}\frac{\rho^{k+1}+\rho^k}{2(\rho^k)^2}\left(\|\Lambda_1^{k+1} - \Lambda_1^k\|_F^2 + \|\Lambda_2^{k+1} - \Lambda_2^k\|_F^2\right)\\
  &\leq& +\infty.
 \end{eqnarray*}

Since $\mL(P^k;D^k;\rho^k)$ is bounded, then it follows  that

\begin{equation}
0\leq \sum_{k=0}^{K-1}c^k\bigg(\|{Z}^{k+1}-{Z}^k\|_F^2+\|{X}^{k+1}-{X}^k\|_F^2
+\|{Y}^{k+1}-{Y}^k\|_F^2 \bigg)\leq +\infty.
\end{equation}
Since additionally $\sum_k \rho_k = +\infty$, this
 immediately yields ${Z}^{k+1}-{Z}^k\to 0, {X}^{k+1}-{X}^k\to 0, {Y}^{k+1}-{Y}^k\to 0$.

 Moreover,
 \[
 \sum_{k=1}^\infty \frac{1}{\rho^k} < \infty \iff \frac{1}{\rho^k}\to 0  \iff \rho^k\to \infty.
 \]
 Combined with
  Assumption \ref{Assumption:BoundedDualXxy}, this implies that
 \[
 Z^{k+1}-X^{k+1}(Y^{k+1})^T = \frac{1}{\rho^k}(\Lambda_1^{k+1} - \Lambda_1^k)\to 0,
 \]
 \[
 X^{k+1} - Y^{k+1} = \frac{1}{\rho^k}(\Lambda_2^{k+1} - \Lambda_2^k) \to 0.
 \]
 Therefore the limit points $X^*, Y^*$, and $Z^*$ are all feasible, and simply checking the first optimality condition will verify that this accumulation point is a stationary point of (\ref{eq:FormulationXxy1}). The proof is complete.

\end{proof}

\end{appendix}

\end{document}